\documentclass[12pt]{article}
\usepackage{amsmath,amssymb,amsbsy,amsfonts,amsthm,latexsym,amsopn,amstext,
                            amsxtra,euscript,amscd}
\usepackage{booktabs}
\usepackage{tikz}
\newtheorem{theorem}{Theorem}

\newtheorem{prop}[theorem]{Proposition}

\newtheorem{thm}[theorem]{Theorem}

\newtheorem{lem}[theorem]{Lemma}

\DeclareMathOperator{\card}{card}

\def\\{\cr}
\def\({\left(}
\def\){\right)}
\def\[{\left[}
\def\]{\right]}
\def\<{\langle}
\def\>{\rangle}

\def\N{\mathbb{N}}
\def\C{\mathbb{C}}

\def\notdivides{\mathrel{\kern-3pt\not\!\kern3.5pt\bigm|}}

\begin{document}


\title{Counting tuples restricted by coprimality conditions}
\author{%
{\sc Juan Arias de Reyna}\\
{Department of Mathematical Analysis, Seville University} \\
{Seville, Spain}\\
{\tt arias@us.es}
\and
{\sc Randell Heyman}  \\
School of Mathematics and Statistics,\\ University of New South Wales \\
Sydney, Australia\\
{\tt randell@unsw.edu.au}
}

\maketitle

\begin{abstract}
Given a set $A=\{(i_1,j_1),\ldots,(i_m,j_m)\}$ we say that $(a_1,\ldots,a_v)$ exhibits \emph{pairwise coprimality} if $\gcd(a_i,a_j) = 1$ for all $(i,j)\in A$.
For a given positive $x$ we give an asymptotic formula for the number of $(a_1,\ldots,a_v)$ with $1 \le a_1,\ldots,a_v \le x$ that exhibit pairwise coprimality. Our error term is better than that of Hu.
\end{abstract}

\noindent
\section{Introduction}
We study tuples whose elements are positive integers of maximum value $x$ and impose certain coprimality conditions on pairs of elements.
T$\acute{\textrm{o}}$th \cite{Tot} used an inductive approach to give an asymptotic formula for the number of height constrained tuples that exhibit pairwise coprimality. For a generalisation from pairwise coprimality to $v$-wise coprimality see \cite{Hu1}.

Recently Fern$\acute{\rm{a}}$ndez and Fern$\acute{\rm{a}}$ndez, in \cite{Fer} and in subsequent discussions with the second author, have shown how to calculate the probability that $v$ positive integers of any size exhibit coprimality across given pairs.  Their approach is non-inductive.
Hu \cite{Hu2} has estimated the number of $(a_1,\ldots,a_v)$ with $1 \le a_1,\ldots,a_v \le x$  that satisfy given coprimality conditions on pairs of elements of the $v$-tuple. His inductive approach gives an asymptotic formula with an upper bound on the error term of $O(x^{v-1}\log^{v-1}x)$.

Coprimality across given pairs of elements of a $v$-tuple is not only interesting in its own right. To date it has been necessary for quantifying $v$-tuples that are \emph{totally pairwise non-coprime}, that is, $\gcd(i,j)>1$ for all $1 \le i,j \le v$ (see \cite{Hu2},\cite{Hey} and \cite{Mor} and its comments regarding \cite{Fre}).

Our main result gives a better error term than that of \cite{Hu2}. Unlike \cite{Hu2} our approach is non-inductive.

We use a graph to represent the required primality conditions as follows. Let $G=(V,E)$ be a graph with $v$ vertices and $e$ edges. The set of vertices, $V$, will be given by $V=\{1,\ldots,v\}$ whilst the set of edges of $G$, denoted by $E$, is a subset of the set of pairs of elements of $V$. That is, $E \subset \{\{1,2\},\{1,3\},\ldots,\{r,s\},\ldots,\{v-1,v\}\}$. We admit isolated vertices (that is, vertices that are not adjacent to any other vertex). An edge is always of the form $\{r,s\}$ with $r \ne s$ and $\{r,s\}=\{s,r\}$.
For each real $x>0$ we define the set of all tuples that satisfy the primality conditions by
$$G(x):=\{(a_1,\ldots,a_v) \in \mathbb{N}^v: a_r \le x,~\gcd(a_r,a_s)=1~\text{if}~\{r,s\}\in E\}.$$
We also let $g(x)=\card (G(x))$, and denote with $d$ the maximum degree of the vertices of $G$. Finally, let $Q_G(x)=1+a_2x^2+\cdots + a_vx^v$ be the polynomial associated to the graph $G$ defined in Section \ref{Preparations}.

Our main result is as follows.
\begin{thm}\label{main}

For real $x>0$ we have
$$g(x)=x^v \rho_G+O(x^{v-1}\log^d x),$$
where
$$\rho_G=\prod_{p~\textrm{prime}}Q_G\(\frac{1}{p}\).$$
\end{thm}

\section{Preparations}\label{Preparations}
As usual, for any integer $n\ge 1,$ let $\omega(n)$ and $\sigma(n)$ be the number of distinct prime factors of $n$ and the sum of divisors of $n$
respectively (we also set $\omega(1) =0$). We also use $\mu$ to denote the M{\" o}bius function, that is,
$\mu(n)=(-1)^{\omega(n)}$ if $n$ is square free, and $\mu(n)=0$ otherwise.
$P^+(n)$ denotes the largest prime factor of the integer $n>1$. By convention $P^+(1)=1$.
We recall that the notation $U = O(V)$  is
equivalent to the assertion that the inequality $|U|\le c|V|$ holds for some
constant $c>0$. We will denote the least common multiple of integers $x_1,\ldots,x_v$ by $[x_1,\ldots,x_v]$.

For each $F \subset E$, a subset of the edges of $G$,  let $v(F)$ be the number of non-isolated vertices of $F$. We define two polynomials $Q_G(x)$ and $Q_G^+(x)$   by
$$Q_G(x)=\sum_{F \subset E} (-1)^{\card(F)}x^{v(F)},\qquad
Q_G^+(x)=\sum_{F \subset E} x^{v(F)}.$$
In this way we associate two polynomials to each graph. It is clear that the only $F\subset E$ for which $v(F)=0$ is the empty set. Thus the constant  term of $Q_G(x)$ and $Q_G^+(x)$
is always 1. If $F$ is non-empty then there is some edge $a=\{r,s\} \in F$ so that $v(F) \ge 2$. Therefore the coefficient of $x$ in $Q_G(x)$ and $Q_G^+(x)$  is zero.
Since we do not allow repeated edges the only case in which $v(F)=2$ is when $F$ consists of one edge. Thus the coefficient of $x^2$ in $Q_G^+(x)$ is $e$, that is, the number of edges $e$ in $G$. The
corresponding $x^2$ coefficient in $Q_G(x)$ is $-e$.

As a matter of notation we shall sometimes use $r$ and $s$ to indicate vertices. The letter $v$ will always denote the last vertex and the number of vertices in a given graph. Edges will sometimes be denoted by $a$ or $b$. As previously mentioned, we use $d$ to denote the maximum degree of any vertex and $e$ to denote the number of edges.  We use terms like $e_j$ to indicate the $j$-th edge.

We associate several multiplicative functions to any graph. To define these functions we consider  functions  $E \rightarrow \mathbb{N}$,
that is, to any edge $a$ in the graph we associate a natural number $n_a$.
We call any of these functions, $a\mapsto n_a$, an
\emph{edge numbering} of the graph. Given an edge numbering we assign a corresponding
\emph{vertex numbering} function $r\mapsto N_r$ by the rule  $N_r=[n_{b_1},\ldots,n_{b_u}]$,
where $E_r=\{b_1,\ldots,b_u\}\subset E$ is the set of edges incident to $r$.
We note that in the case where $r$ is an isolated vertex we will have
$E_r=\emptyset$ and $N_r=1$. With these notations we define
$$f_G(m)=\sum_{N_1N_2\cdots N_v=m} \mu(n_1)\cdots\mu(n_e),\quad
f^+_G(m)=\sum_{N_1N_2\cdots N_v=m} |\mu(n_1)\cdots\mu(n_e)|,$$
where the sums extend to all possible edge numberings of $G$.

The following is interesting in its own right but will also be used to prove Theorem \ref{main}.
\begin{prop}\label{multi}
Let $f:\N \rightarrow \C$ be a multiplicative function. For any graph $G$ the function
$$g_{f,G}(m)=\sum_{N_1N_2\cdots N_v=m} f(n_1)\cdots f(n_e)$$
is multiplicative.
\end{prop}
\begin{proof}
Let $m=m_1m_2$ where $\gcd(m_1,m_2)=1$. Let us assume that for a given edge numbering of $G$ we have $N_1\cdots N_v=m$. For any edge $a=\{r,s\}$ we have $n_a|N_r$ and $n_a|N_s$. Therefore $n_a^2|m$. It follows that we may express $n_a$ as $n_a=n_{1,a}n_{2,a}$ with $n_{1,a}|m_1$ and $n_{2,a}|m_2$. In this case $\gcd(n_{1,a},n_{2,a})=1$, and we will have
$$N_r=[n_{b_1},\ldots,n_{b_v}]=[n_{1,b_1},\ldots,n_{1,b_v}][n_{2,b_1},\ldots,n_{2,b_v}],$$
$$f(n_1)\cdots f(n_e)=f(n_{1,1})\cdots f(n_{1,e})\cdot f(n_{2,1})\cdots f(n_{2,e}).$$
Since each edge numbering $n_a$ splits into two edge numberings $n_{1,a}$ and $n_{2,a}$, we have
$$m_1=N_{1,1}\cdots N_{1,v}, \quad m_2=N_{2,1}\cdots N_{2,v}.$$
Thus
\begin{align*}
g_{f,G}(m_1m_2)&=g_{f,G}(m)\\
&=\sum_{N_1N_2\cdots N_v=m} f(n_1)\cdots f(n_e)\\
&=\sum_{N_{1,1}\cdots N_{1,v}\cdot N_{2,1}\cdots N_{2,v}=m_1m_2}f(n_{1,1})\cdots f(n_{1,e})\cdot f(n_{2,1})\cdots f(n_{2,e}) \\
&=\sum_{N_{1,1}\cdots N_{1,v}=m_1}f(n_{1,1})\cdots f(n_{1,e})\sum_{N_{2,1}\cdots N_{2,v}=m_2}f(n_{2,1})\cdots f(n_{2,e}) \\
&=g_{f,G}(m_1)g_{f,G}(m_2),
\end{align*}
which completes the proof.
\end{proof}
We now draw the link between $f^+_G(p^k)$ and $Q_G^+(x)$.
\begin{lem}\label{L:functions}
For any graph $G$ and prime $p$ the value $f^+_G(p^k)$ is equal to the coefficient of $x^k$ in $Q_G^+(x)$. In the same way the value of $f_G(p^k)$ is equal to the coefficient of $x^k$
in $Q_G(x)$.
\end{lem}

\begin{proof}
First we consider the case of $f_G(p^k)$.
Recall that
$$Q_G(x)=\sum_{F\subset E}(-1)^{\card(F)}x^{v(F)},\qquad
f_G(p^k)=\sum_{N_1\cdots N_v=p^k}\mu(n_1)\cdots \mu(n_e),$$
where the last sum is on the set of edge numberings of $G$.
In the second sum we shall only consider edge numberings of $G$ giving a non null term.
This means that we only consider edge numberings with $n_a$ squarefree numbers.
Notice also that if $N_1\cdots N_v=p^k$,
then each $n_a\mid p^k$. So the second sum extends to all
edge numbering with $n_a\in \{1,p\}$ for each edge $a\in E$ and satisfying $N_1\cdots N_v=p^k$.

We need to prove the equality
\begin{equation}\sum_{F\subset E,\  v(F)=k}(-1)^{\card(F)}=\sum_{N_1\cdots N_v=p^k}\mu(n_1)\cdots \mu(n_e).
\label{E:lemma3}\end{equation}
To this end we shall define for each $F\subset E$ with $v(F)=k$ a squarefree edge numbering
$\sigma(F)=(n_a)$ with
$N_1\cdots N_v=p^k$,
$n_a\in\{1,p\}$ and such that $(-1)^{\card(F)}=\mu(n_1)\cdots \mu(n_e)$.
We will show that $\sigma$ is a bijective mapping between the set of $F\subset E$ with $v(F)=k$
and the set of edge numberings $(n_a)$ with $N_1\cdots N_v=p^k$.  Thus equality \eqref{E:lemma3}
will be established and the proof finished.

Assume that  $F\subset E$ with $v(F)=k$. We define $\sigma(F)$ as the edge numbering $(n_a)$
defined by
$$n_a=p \text{ for any $a\in F$}, \quad n_a=1\text{ for $a\in E\smallsetminus F$}.$$
In this way it
is clear that $\mu(n_1)\cdots \mu(n_e)=(-1)^{\card(F)}$. Also $N_r=p$ or $N_r=1$.
We will have $N_r=p$ if and only if there is some $a=\{r,s\}\in F$. So that
$N_1\cdots N_v=p^{v(F)}$ because by definition $v(F)$ is the cardinality of the union
$\bigcup_{\{r,s\}\in F}\{r,s\}$.

The map $\sigma$ is  invertible. For let $(n_a)$ be an edge numbering of squarefree numbers
with $N_1\cdots N_v=p^k$ and $n_a\in\{1,p\}$. If $\sigma(F)=(n_a)$ necessarily we will have
$F=\{a\in E: n_a=p\}$. It is clear that defining $F$ in this way we will have
$v(F)=k$ and  $\sigma(F)=(n_a)$.

Therefore the coefficient of $x^k$ in $Q_G(x)$ coincide with the value of $f_G(p^k)$.

The proof for $f_G^+$ is the same  observing that for $\sigma(F)=(n_a)$ we will have
$1=|(-1)^{\card(F)}|=|\mu(n_1)\cdots \mu(n_e)|$.
\end{proof}

\section{Proof of Theorem \ref{main}}
We prove the theorem in the following steps:
\begin{enumerate}
\item We show that
$$g(x)=\sum_{n_1,\ldots,n_e}\mu(n_1)\cdots\mu(n_e)\prod_{r=1}^v \bigg\lfloor\frac{x}{N_{r}}\bigg\rfloor.$$
\item
We show that
$$g(x)=x^v \sum_{n_1=1}^\infty \cdots \sum_{n_e=1}^\infty \mu(n_1)\cdots \mu(n_e)\prod_{r=1}^v \frac{1}{N_{r}}+R+O\(x^{v-1}\log^d x\),$$
where
$$|R| \le x^{v-1} \sum_{j=1}^e \sum_{n_1=1}^\infty\cdots\sum_{n_{j-1}=1}^\infty\sum_{n_j>x}\sum_{n_{j+1}=1}^\infty\cdots \sum_{n_e=1}^\infty\mu(n_1)\cdots \mu(n_e)\prod_{r=1}^v \frac{1}{N_{r}}.$$
\item We show that $|R|=O(x^{v-1}\log^d x)$.
\end{enumerate}
We start with the following sieve result which generalises the sieve of Eratosthenes.

\begin{lem}\label{sieve}
Let $X$ be a finite set, and let $A_1,A_2,\ldots,A_k \subset X$. Then
$$\card\(X\backslash \bigcup_{j=1}^k A_j\)=\sum_{J \subset \{1,2,\ldots,k\}}(-1)^{\card(J)}\card(A_J),$$
where $A_\emptyset=X$, and for $J\subset \{1,2,\ldots,k\}$ nonempty
$$A_J=\bigcap_{j \in J} A_j.$$
\end{lem}
To prove Theorem \ref{main} let $X$ be the set
$$X=\{(a_1,\ldots,a_v) \in \mathbb{N}^v: a_r\le x, 1\le r\le v\}.$$
Our set $G(x)$, associated to the graph $G$, is a subset of $X$. Now for each prime $p\le x$ and each edge $a=\{r,s\} \in G$ define the following subset of X.
$$A_{p,a}=\{(a_1,\ldots,a_v) \in X : p|a_r,p|a_s\}.$$
Therefore the tuples in $A_{p,a}$ are not in $G(x)$. In fact it is clear that
$$G(x)=X\backslash \bigcup_{\substack{a\in E\\p \le x}}A_{p,a},$$
where $E$ denotes the set of edges in our graph $G$.
We note that we have an $A_{p,a}$ for each prime number
less than or equal to $x$ and each edge $a \in E$. Denoting $P_x$ as the set of prime numbers less than or equal to $x$ we can represent each $A_{p,a}$ as $A_j$ with $j \in P_x \times E$.
We now apply Lemma \ref{sieve} and obtain
\begin{align}\label{cardgx}
g(x)=\sum_{J \subset P_x \times E} (-1)^{\card(J)}\card(A_J).
\end{align}

We compute $\card(A_J)$ and then $\card(J)$. For $\card(A_J)$ we have
$$J=\{(p_1,e_1),\ldots,(p_m,e_m)\},\quad A_J=\bigcap_{j=1}^m A_{p_j,e_j}.$$
Therefore $(a_1,\ldots,a_v) \in A_J$ is equivalent to saying that $p_j|a_{r_j},p_j|a_{s_j}$ for all $1 \le j \le m$, where $e_j=\{r_j,s_j\}$.
We note that if $p_{i_1},\ldots, p_{i_\ell}$ are the primes associated in $J$ with a given
edge $a=\{r,s\}$, then the product of $p_{i_1} \cdots p_{i_\ell}$ must also divide
the values $a_r$ and $a_s$ associated to the vertices of $a$.  Let $T_a \subset P_x$ consist of the primes $p$ such that $(p,a) \in J$. In addition we define
$$n_a=\prod_{p \in T_{a}}p, $$
observing that when $T_a=\emptyset$ we have $n_a=1$.
Then $(a_1,\ldots,a_v) \in A_J$ is equivalent to saying that for each $a=\{r,s\}$ appearing in $J$
we have $n_a\mid a_r$ and $n_a\mid a_s$.
In this way we can define $J$ by giving a number $n_a$ for each edge $a$. We note that $n_a$ will always be squarefree,  and all its prime factors will be less than or equal to $x$. We also note that $(a_1,\ldots,a_v) \in A_J$ is equivalent to saying that $n_a|a_r$ for each edge $a$ that joins vertex $r$ with another vertex.

Then for each vertex $r$, consider all the edges $a$ joining $r$ to other vertices, and denote the least common multiple of the corresponding $n_a$'s by $N_{r}$. So $(a_1,\ldots,a_v) \in A_J$ is equivalent to saying that $N_{r}|a_r$. The number of multiples of $N_{r}$ that are less than or equal to $x$ is $\lfloor x/N_{r} \rfloor$, so we can express the number of elements of $A_J$ as
\begin{align}\label{cardaj}
\card(A_J)=\prod_{r=1}^v \bigg\lfloor\frac{x}{N_{r}}\bigg\rfloor.
\end{align}

We now compute $\card(J)$. This is the total number of prime factors across all the $n_j$. As mentioned before $n_j$ is squarefree, so
\begin{align}\label{cardj}
(-1)^{\card(J)}=(-1)^{\sum_{\substack{j=1}}^e\omega(n_j)}=\mu(n_1)\cdots\mu(n_e),
\end{align}
where the summations are over all squarefree $n_j$ with $P^+(n_j) \le x$.
Substituting \eqref{cardaj} and \eqref{cardj} into \eqref{cardgx} yields
\begin{align*}
g(x)=\sum_{n_1=1}^\infty\cdots \sum_{n_e=1}^\infty\mu(n_1)\cdots\mu(n_e)\prod_{r=1}^v \bigg\lfloor\frac{x}{N_{r}}\bigg\rfloor.
\end{align*}
At first the sum extends to the $(n_1,\dots, n_e)$ that are squarefree and have
all prime factors less than or equal to $x$. But we may extend the sum to all $(n_1,\dots, n_e)$,
because if these conditions are not satisfied then the corresponding term is automatically $0$.
In fact we may restrict the summation to the $n_a\le x$, because otherwise
for $a=\{r,s\}$ we have $n_a\mid N_{r}$ and $\lfloor x/N_{r}\rfloor=0$. Therefore
\begin{align*}
g(x)=\sum_{1 \le n_1 \le x}\cdots \sum_{1 \le n_e\le x}\mu(n_1)\cdots\mu(n_e)\prod_{r=1}^v \bigg\lfloor\frac{x}{N_{r}}\bigg\rfloor.
\end{align*}
We now seek to express $g(x)$ as a multiple of $x^v$ plus a suitable error term.
Observe that for all real $z_1,z_2,z_3>0$,
$$\lfloor z_1\rfloor\lfloor z_2\rfloor\lfloor z_3\rfloor=
z_1z_2z_3-z_1z_2\{z_3\}-z_1\{z_2\}\lfloor z_3\rfloor-\{z_1\}\lfloor z_2\rfloor\lfloor z_3\rfloor,$$
where $\{y\}$ denotes the fractional part of a number $y$.

Applying a similar procedure, with $v$ factors instead of $3$, we get
\begin{align}\label{mainterm and error}
g(x)&=\sum_{1 \le n_1 \le x}\cdots \sum_{1 \le n_e\le x}\mu(n_1)\cdots\mu(n_e)\prod_{r=1}^v \frac{x}{N_{r}}\notag\\
&-\sum_{1 \le n_1 \le x}\cdots \sum_{1 \le n_e\le x}\mu(n_1)\cdots\mu(n_e)\bigg\{\frac{x}{N_{1}}\bigg\}\prod_{r=2}^v \bigg\lfloor\frac{x}{N_{r}}\bigg\rfloor\notag\\
&-\sum_{1 \le n_1 \le x}\cdots \sum_{1 \le n_e\le x}\mu(n_1)\cdots\mu(n_e)\frac{x}{N_{1}}\bigg\{\frac{x}{N_{2}}\bigg\}\prod_{r=3}^v \bigg\lfloor\frac{x}{N_{r}}\bigg\rfloor\notag\\
&\cdots\notag\\
&-\sum_{1 \le n_1 \le x}\cdots \sum_{1 \le n_e\le x}\mu(n_1)\cdots\mu(n_e)\frac{x}{N_{1}}\cdots \frac{x}{N_{v-1}}\bigg\{\frac{x}{N_v}\bigg\}\notag\\
&=x^v\sum_{1 \le n_1 \le x}\cdots \sum_{1 \le n_e\le x}\mu(n_1)\cdots\mu(n_e)\prod_{r=1}^v \frac{1}{N_{r}}+
\sum_{k=1}^vR_k,
\end{align}
where for $1 \le k \le v$,
\begin{align*}
R_k&=-\sum_{1 \le n_1 \le x}\cdots \sum_{1 \le n_e\le x}\mu(n_1)\cdots\mu(n_e)\frac{x}{N_{1}}\cdots\frac{x}{N_{k-1}} \bigg\{\frac{x}{N_{k}}\bigg\}\bigg\lfloor\frac{x}{N_{k+1}}\bigg\rfloor\cdots\bigg\lfloor\frac{x}{N_{{v}}}\bigg\rfloor,
\end{align*}
with the obvious modifications for $j=1$ and $j=v$.
We then have
\begin{align*}
|R_k|&\le \sum_{1 \le n_1 \le x}\cdots \sum_{1 \le n_e\le x}|\mu(n_1)\cdots\mu(n_e)|\frac{x}{N_{1}}\cdots\frac{x}{N_{k-1}} \frac{x}{N_{k+1}}\cdots\frac{x}{N_{{v}}}\\
&\le x^{v-1}\sum_{P^+(m)\le x} \frac{C_{G,k}(m)}{m},
\end{align*}
where
$$C_{G,k}(m)=\sum_{m=\prod_{1 \le r \le v, r \ne k}N_{r}}|\mu(n_1)\cdots\mu(n_e)|.$$
By similar reasoning to that of Proposition \ref{multi} the function
$C_{G,k}(m)$ can be shown to be multiplicative.
The numbers $C_{G,k}(p^{\alpha})$
do not depend on $p$, and $C_{G,k}(p^{\alpha})=0$
for $\alpha>v$. So we have
\begin{align*}
\sum_{P^+(m)\le x} \frac{C_{G,k}(m)}{m}&\le\prod_{p \le x}\(1+\frac{C_{G,k}(p)}{p}
+\frac{C_{G,k}(p^2)}{p^2}+\cdots\frac{C_{G,k}(p^v)}{p^v}\)\\
&=O(\log^{C_{G,k}(p)}x),
\end{align*}
where $C_{G,k}(m)$ is the number of solutions $(n_1,\dots, n_e)$, with $n_j$ squarefree,  to
\begin{align}
\prod_{1 \le r \le v, r\ne k}N_{r}=m.
\end{align}
Let $h_k$ denote the degree of
vertex $k$. It is easy to see that for a prime $p$ we have $C_{G,k}(p)=h_k$.  The solutions are precisely those with all $n_j=1$, except one  $n_\ell=p$, where
$\ell$ should be one of the edges meeting at vertex $k$.
Therefore the maximum number of solutions occurs when $k$ is one of the vertices of maximum degree. So if we let $d$ be this maximum degree, then the maximum value of $C_{G,k}(p)$ is $d$. Therefore
\begin{align}\label{firsterror}
|R_k| = O(x^{v-1}\log^{d}x).
\end{align}
Substituting \eqref{firsterror} into \eqref{mainterm and error} we obtain
\begin{align}\label{afterfirsterror}
g(x)=x^v\sum_{1 \le n_1 \le x}\cdots \sum_{1 \le n_e\le x} \mu(n_1)\cdots\mu(n_e)\prod_{r=1}^v \frac{1}{N_{r}}+O(x^{v-1}\log^{d}x).
\end{align}
We require the following lemma.
\begin{lem}\label{absolute}
$$\lim_{x \rightarrow \infty}\sum_{1 \le n_1 \le x}\cdots \sum_{1 \le n_e\le x} |\mu(n_1)\cdots\mu(n_e)|\prod_{r=1}^v \frac{1}{N_{r}}<+\infty.$$
\end{lem}
\begin{proof}
We have
\begin{align}\label{sumb1}
\lim_{x \rightarrow \infty}\sum_{1 \le n_1 \le x}\cdots \sum_{1 \le n_e\le x} |\mu(n_1)\cdots \mu(n_e)|\prod_{r=1}^v \frac{1}{N_{r}}
&=\sum_{m=1}^\infty \frac{f^+_G(m)}{m},
\end{align}
where
$$f^+_G(m)=\sum_{m=\prod_{r=1}^vN_{r}}|\mu(n_1)\cdots \mu(n_e)|.$$
We note that $f^+_G(m)$ is multiplicative by Proposition \ref{multi}.
It is clear that $f_G^+(1)=1$. Also, each edge joins two vertices $r$ and $s$ and thus $n_j|E_r$ and $n_j|E_s$. This means that
$$n_j^2 \big|\prod_{r=1}^v N_{r}.$$ It follows that
$$\prod_{r=1}^v N_{r}\ne p,$$ for any prime $p$ and so $f^+_G(p)=0.$
We also note that a multiple $(n_1,\ldots,n_e)$ only counts in $f^+_G(m)$ if $|\mu(n_1)\cdots \mu(n_e)|=1$. Therefore each $n_j$ is squarefree. So each factor in
\begin{align}\label{prod}
\prod_{r=1}^v N_{r}
\end{align}
brings at most a $p$. So the greatest power of $p$ that can divide \eqref{prod} is $p^v$. So $f^+_G(p^{\alpha})=0$ for $\alpha>v$.
Recall that $f^+_G(p^\alpha)$ is equal to the coefficient of $x^\alpha$ in  $Q_G^+(x)$. So, by Lemma \ref{L:functions}, we note that  $f^+_G(p^\alpha)$ depends on $\alpha$ but not on $p$.
Putting all this together we have
\begin{align}\label{sumb2}
\sum_{m=1}^\infty \frac{f^+_G(m)}{m}=\prod_{p~\textrm{prime}}\(1+\frac{f^+_G(p^2)}{p^2}+\ldots +\frac{f^+_G(p^v)}{p^v}\)< +\infty.
\end{align}
Substituting \eqref{sumb2} into \eqref{sumb1} completes the proof.
\end{proof}
Returning to \eqref{afterfirsterror} it is now clear from Lemma \ref{absolute} that
$$\rho_G=\lim_{x \rightarrow \infty}\sum_{1 \le n_1 \le x}\cdots \sum_{1 \le n_e\le x} \mu(n_1)\cdots\mu(n_e)\prod_{r=1}^v \frac{1}{N_{r}}$$ is absolutely convergent.
In fact,
\begin{align}\label{gx}
g(x)=x^v \rho_G + R + O(x^{v-1} \log^d x),
\end{align}
where
$$\rho_G=\sum_{n_1=1 }^\infty\cdots \sum_{n_e=1}^\infty \mu(n_1)\cdots\mu(n_e)\prod_{r=1}^v \frac{1}{N_{r}},$$
and
$$|R| \le x^{v-1} \sum_{j=1}^e \sum_{n_1=1}^\infty\cdots\sum_{n_{j-1}=1}^\infty\sum_{n_j>x}\sum_{n_{j+1}=1}^\infty\cdots \sum_{n_e=1}^\infty|\mu(n_1)\cdots \mu(n_e)|\prod_{r=1}^v \frac{1}{N_{r}}.$$
Now
$$\rho_G=\sum_{m=1}^\infty \frac{1}{m}\sum_{N_1\cdots N_v=m}\mu(n_1)\cdots \mu(n_e)=\sum_{m=1}^\infty \frac{f_G(m)}{m}.$$
We note that $f_G(m)$ is multiplicative by Proposition \ref{multi}. In a similar way to Lemma \ref{absolute} we have $f_G(1)=1, f_G(p)=0$ and $f_G(p^{\alpha})=0$, for all $\alpha>v$. Thus,
by the multiplicativity,
\begin{align*}
\rho_G=\sum_{m=1}^\infty \frac{f_G(m)}{m}=\prod_{p~\textrm{prime}}\(1+\frac{f_G(p^2)}{p^2}+\ldots+\frac{f_G(p^v)}{p^v}\),
\end{align*}
Therefore, by Lemma \ref{L:functions}, we have
\begin{align}\label{rhog}
\rho_G=\prod_{p~\textrm{prime}}Q_G\(\frac{1}{p}\).
\end{align}
Substituting \eqref{rhog} into \eqref{gx}, it only remains to show that $|R|=O(x^{v-1}\log^d x)$.

We have
\begin{align*}
|R| &\le x^{v-1} \sum_{j=1}^e \sum_{n_1=1}^\infty\cdots\sum_{n_{j-1}=1}^\infty\sum_{n_j>x}\sum_{n_{j+1}=1}^\infty\cdots \sum_{n_e=1}^\infty|\mu(n_1)\cdots \mu(n_e)|\prod_{r=1}^v \frac{1}{N_{r}}.
\end{align*}
All terms in the sum on $j$ are analogous; so assuming that the
first is the largest, we have
\begin{align*}
|R|&\le C_1x^{v-1}\sum_{n_1>x}\sum_{n_2=1}^\infty\sum_{n_{j+1}=1}^\infty\cdots \sum_{n_e=1}^\infty|\mu(n_1)\cdots \mu(n_e)|\prod_{r=1}^v \frac{1}{N_{r}},
\end{align*}
where
$C_1$ is a function of $e$ and not $x$. So it will suffice to show that
\begin{align}\label{R1}
R_1:=\sum_{n_1>x}\sum_{n_2=1}^\infty\cdots \sum_{n_e=1}^\infty|\mu(n_1)\cdots \mu(n_e)|\prod_{r=1}^v \frac{1}{N_{r}}=O(\log^d x).
\end{align}
We will treat an edge $e_1=\{r,s\}$ differently to the other edges.
For a given $(n_1,\ldots,n_e)$ of squarefree numbers  we have two special $N_r$,
$$N_r=[n_1,n_{\alpha_1},\ldots n_{\alpha_k}], \quad  N_s=[n_1,n_{\beta_1},\ldots n_{\beta_k}].$$
We also remark that we may have $N_r=[n_1]$ or $N_s=[n_1]$.

For any edge $e_j$ with $2 \le j \le e$ we define $d_{j}=\gcd(n_1,n_{j})$.
Since the $n_j$ are squarefree, we have
$$n_{j}=d_{j}n'_{j},\quad d_{j}|n_1, \quad \gcd(n_1,n'_{j})=1.$$
Then it is clear that
$$N_{r}=[n_1,d_{\alpha_1}n'_{\alpha_1},\ldots,d_{\alpha_k}n'_{\alpha_k}]=n_1[n'_{\alpha_1},\ldots,n'_{\alpha_k}],\quad N_{s}=n_1[n'_{\beta_1},\ldots,n'_{\beta_l}].$$
For any other vertex with $t \ne r$ and  $t\ne s$, we have
$$N_{t}=[n_{t_1},\ldots,n_{t_m}]=[d_{t_1}n'_{t_1},\ldots,d_{t_m}n'_{t_m}]=[d_{t_1},\ldots,d_{t_m}][n'_{t_1},\ldots,n'_{t_m}],$$
where $m$ will vary with $t$.
Substituting the equations for $N_{r}, N_{s}$ and $N_{t}$ into the definition of $R_1$ in \eqref{R1} we obtain
\begin{align*}
R_1&=\sum_{n_1>x}\sum_{n_2=1}^\infty\cdots \sum_{n_e=1}^\infty|\mu(n_1)\cdots \mu(n_e)|\frac{1}{N_r}\frac{1}{N_s}\prod_{\substack{1 \le t \le v\\t\ne r,~t \ne s}} \frac{1}{N_{t}}\\
&=\sum_{n_1>x}\frac{|\mu(n_1)|}{n_1^2}\sum_{d_2|n_1}\cdots \sum_{d_e|n_1}\sum_{n'_2=1}^\infty\cdots \sum_{n'_e=1}^\infty\frac{|\mu(n_2)\cdots \mu(n_e)|}{[n'_{\alpha_1},\ldots,n'_{\alpha_k}][n'_{\beta_1},\ldots,n'_{\beta_l}]}\\&\quad\quad\quad\quad \times\prod_{\substack{1 \le t \le v\\t\ne r,~t \ne s}} \frac{1}{[d_{t_1},\ldots,d_{t_m}][n'_{t_1},\ldots,n'_{t_m}]}\\
&=\sum_{n_1>x}\frac{|\mu(n_1)|}{n_1^2}\sum_{d_2|n_1}\cdots \sum_{d_e|n_1}
\prod_{\substack{1 \le t \le v\\t\ne r,~t \ne s}} \frac{1}{[d_{t_1},\ldots,d_{t_m}]}\\&\quad\quad\quad\quad\times    \sum_{n'_2=1}^\infty\cdots \sum_{n'_e=1}^\infty\frac{|\mu(d_2n'_2)\cdots \mu(d_en'_e)|}{[n'_{\alpha_1},\ldots,n'_{\alpha_k}][n'_{\beta_1},\ldots,n'_{\beta_l}]} \prod_{\substack{1 \le t \le v\\t\ne r,~t \ne s}} \frac{1}{[n'_{t_1},\ldots,n'_{t_m}]}\\
&\le\sum_{n_1>x}\frac{|\mu(n_1)|}{n_1^2}\sum_{d_2|n_1}\cdots \sum_{d_e|n_1}
|\mu(d_2)\cdots\mu(d_e)|
\prod_{\substack{1 \le t \le v\\t\ne r,~t \ne s}}
\frac{1}{[d_{t_1},\ldots,d_{t_m}]}\\&\quad\quad\quad\quad\times    \sum_{n'_2=1}^\infty\cdots \sum_{n'_e=1}^\infty\frac{|\mu(n'_2)\cdots \mu(n'_e)|}{[n'_{\alpha_1},\ldots,n'_{\alpha_k}][n'_{\beta_1},\ldots,n'_{\beta_l}]} \prod_{\substack{1 \le t \le v\\t\ne r,~t \ne s}} \frac{1}{[n'_{t_1},\ldots,n'_{t_m}]}.
\end{align*}
The product
$$\sum_{n'_2=1}^\infty\cdots \sum_{n'_e=1}^\infty\frac{|\mu(n'_2)\cdots \mu(n'_e)|}{[n'_{\alpha_1},\ldots,n'_{\alpha_k}][n'_{\beta_1},\ldots,n'_{\beta_l}]} \prod_{\substack{1 \le t \le v\\t\ne r,~t \ne s}} \frac{1}{[n'_{t_1},\ldots,n'_{t_m}]}$$ is finite
by Lemma \ref{absolute} (but this time considering the graph $G$ without the edge $\{r,s\}$). Thus, for some constant $C_1$, we have
\begin{align}\label{fge}
R_1 &\le C_2 \sum_{n_1>x}\frac{|\mu(n_1)|}{n_1^2}\sum_{d_2|n_1}\cdots \sum_{d_e|n_1}|\mu(d_2)\cdots\mu(d_e)|\prod_{\substack{1 \le t \le v\\t\ne r,~t \ne s}}
\frac{1}{[d_{t_1},\ldots,d_{t_m}]}\notag\\
&=C_2 \sum_{n_1>x}\frac{|\mu(n_1)|}{n_1^2}f_{G,e}(n_1),
\end{align}
where the arithmetic function $f_{G,e}$ is defined as follows.
\begin{align*}
f_{G,e}(n)=\sum_{d_2|n}\cdots \sum_{d_e|n}|\mu(d_2)\cdots\mu(d_e)|
\prod_{\substack{1 \le t \le v\\t\ne r,~t \ne s}} \frac{1}{[d_{t_1},\ldots,d_{t_m}]}.
\end{align*}
We note that there is a factor $[d_{t_1},\ldots,d_{t_m}]$ for each vertex other than $r$ or $s$. The function $f_{G,e}$ is a multiplicative function.
We have $f_{G,e}(p^k)=f_{G,e}(p)$ for any power of a prime $p$ with $k\ge2$, because
in the definition of $f_{G,e}(p^k)$ only the divisors $1$ and $p$ of $p^k$ give
non null terms.
When $n=p$ we have
$$f_{G,e}(p)=1+\frac{A_1}{p}+\cdots +\frac{A_{v-2}}{p^{v-2}},$$
where $A_i$ is the number of ways that
$$\prod_{\substack{1 \le t \le v\\t\ne r,~t \ne s}}|\mu(d_2)\cdots\mu(d_e)|[d_{t_1},\dots, d_{t_m}]=p^i,$$ where every divisor in the product $d_h\mid n=p$ can only be $1$ or  $p$.
Clearly $A_i \le 2^{e-1}$ do not depend on $p$, and so there must be a number $w$,
independent of $p$,  such that
$$f_{G,e}(p^k)=f_{G,e}(p) \le \(1+\frac{1}{p}\)^w.$$ Since $f_{G,e}$
is multiplicative we have,
for any squarefree $n$,
\begin{align}\label{sigma}
f_{G,e}(n)\le \prod_{p|n}\(1+\frac{1}{p}\)^w=\(\frac{\sigma(n)}{n}\)^w, \quad |\mu(n)|=1.
\end{align}
Substituting \eqref{sigma} into \eqref{fge} yields
$$R_1 \le C_2 \sum_{n>x} ^\infty \frac{|\mu(n)|}{n^2}\(\frac{\sigma(n)}{n}\)^w \le C_2 \sum_{n>x} ^\infty \frac{1}{n^2}\(\frac{\sigma(n)}{n}\)^w.$$
It is well known that $\sigma(n)=O(n\log\log n)$ (see, for example, \cite{Gro}),
and thus
\begin{align}\label{loglog}
R_1=O\(\frac{(\log \log x)^w}{x}\).
\end{align}
Comparing \eqref{loglog} with \eqref{R1} completes the proof of Theorem \ref{main}.

\section{Acknowledgement}
The second author would like to thank the first author and his colleagues for their hospitality during a pleasant visit to Seville.


\begin{thebibliography}{9}

\bibitem{Fer} J.~L.~Fern$\acute{\rm{a}}$ndez and P.~Fern$\acute{\rm{a}}$ndez,
`Equidistribution and coprimality',
Preprint, arXiv:1310.3802 [math.NT].
%
\bibitem{Fre}
T.~Freiberg,
`The probability that 3 positive integers are pairwise noprime',
(unpublished manuscript).
%
\bibitem{Gro} T.~H.~Gronwall,
`Some asymptotic expressions in the theory of numbers',
{\it Trans. Amer. Math. Soc.\/} {\bf14}~(1913), 113--122.
%
\bibitem{Har} G.~H.~Hardy and E.~M.~Wright,
{\it An Introduction to the Theory of Numbers (6th Edition)\/}, Oxford Univ. Press,
Oxford, 2008.
%
\bibitem{Hey}
R.~Heyman,
`Pairwise non-coprimality of triples',
Preprint, arxiv: 1309.5578 [math.NT].
%
\bibitem{Hu1}
J.~Hu,
`The probability that random positive integers are $k$-wise relatively prime',
{\it In. J. Number Theory\/} {\bf9}~(2013), 1263--1271.
%
\bibitem{Hu2}
J.~Hu,
`Pairwise relative primality of positive integers',
(unpublished manuscript).
%
\bibitem{Mor}
P.~Moree,
`Counting carefree couples',
Preprint, arxiv:0510003 [math.NT].
%
\bibitem{Nym} J.~E.~Nymann,
`On the probability that $k$ positive integers are relatively prime',
{\it J. Number Theory\/} {\bf4}~(1972), 469--473.
%
\bibitem{Tot} L.~T$\acute{\textrm{o}}$th,
`The probability that $k$ positive integers are pairwise relatively prime',
{\it Fibonacci Quart.\/} {\bf40}~(2002), 13--18.

\end{thebibliography}
\end{document}